\documentclass{aic}
\usepackage{amsmath,amssymb,amsthm,bbm,mathtools,comment}
\usepackage[shortlabels]{enumitem}

\usepackage[numbers,sort]{natbib}
\usepackage{amsfonts}

\usepackage[capitalize]{cleveref}

\theoremstyle{plain}
\newtheorem{theorem}{Theorem}[section]		
\newtheorem{lemma}[theorem]{Lemma}

\newtheorem{proposition}[theorem]{Proposition}
\newtheorem{corollary}[theorem]{Corollary}
\newtheorem{conjecture}[theorem]{Conjecture}

\newtheorem{definition}[theorem]{Definition}

\theoremstyle{remark}
\newtheorem*{remark}{Remark}

\def\E{\mathbb{E}}
\def\EE{\mathcal{E}}

\DeclareMathOperator\ex{ex}
\newcommand{\eps}{\ensuremath{\varepsilon}}
\let\emptyset\varnothing

\newcommand\ab[1]{\lvert #1 \rvert}
\newcommand\cei[1]{\lceil #1 \rceil}
\newcommand\flo[1]{\lfloor #1 \rfloor}

\newcommand\up[1]{^{(#1)}}

\renewcommand{\P}{\mathbb{P}}

\aicAUTHORdetails{%
  title = {Blowups of triangle-free graphs}, 
  author = {Ant\'onio Gir\~ao, Zach Hunter, and Yuval Wigderson},
  plaintextauthor = {Antonio Girao, Zach Hunter, Yuval Wigderson},
}   

\aicEDITORdetails{%
   year={2025},
   number={10},
   received={31 August 2024},   
   published={19 December 2025},  
   doi={10.19086/aic.2025.10},      
}   
\makeatletter
\renewcommand{\toc@licenseurl}{\toc@cclicenseurl} 
\makeatother

\begin{document}
\begin{frontmatter}[classification=text]
\title{Blowups of triangle-free graphs}
\author[antonio]{Ant\'onio Gir\~ao\thanks{Research supported by ERC Advanced Grant no. 883810.}}
\author[zach]{Zach Hunter}
\author[yuval]{Yuval Wigderson\thanks{Research supported by Dr.\ Max R\"{o}ssler, the Walter Haefner Foundation, and the ETH Z\"{u}rich Foundation.}}

\begin{abstract}
    A highly influential result of Nikiforov states that if an $n$-vertex graph $G$ contains at least $\gamma n^h$ copies of a fixed $h$-vertex graph $H$, then $G$ contains a blowup of $H$ of order $\Omega_{\gamma,H}(\log n)$. While the dependence on $n$ is optimal, the correct dependence on $\gamma$ is unknown; all known proofs yield bounds that are polynomial in $\gamma$, but the best known upper bound, coming from random graphs, is only logarithmic in $\gamma$. It is a major open problem to narrow this gap.

    We prove that if $H$ is triangle-free, then the logarithmic behavior of the upper bound is the truth. That is, under the assumptions above, $G$ contains a blowup of $H$ of order $\Omega_H (\log n/{\log(1/\gamma)})$. This is the first non-trivial instance where the optimal dependence in Nikiforov's theorem is known.

    As a consequence, we also prove an upper bound on multicolor Ramsey numbers of blowups of triangle-free graphs, proving that the dependence on the number of colors is polynomial once the blowup is sufficiently large. This shows that, from the perspective of multicolor Ramsey numbers, blowups of fixed triangle-free graphs behave like bipartite graphs.
\end{abstract}
\end{frontmatter}

\section{Introduction}
\subsection{Background and main results}
Given an integer $k$ and a graph $H$, its \emph{blowup} $H[k]$ is the graph obtained from $H$ by replacing every vertex by an independent set of order $k$, and every edge by a copy of the complete bipartite graph $K_{k,k}$. Blowups are fundamental objects in graph theory, and many important results in extremal graph theory concern the problem of finding large $H$-blowups in graphs with certain properties. For example, the Erd\H os--Stone theorem \cite{MR0018807} states that given any graph $H$, integer $k$, and parameter $\varepsilon>0$, any sufficiently large graph $G$ with edge density at least $1-\frac{1}{\chi(H)-1}+\varepsilon$ contains $H[k]$ as a subgraph. Much of the subsequent work in extremal graph theory, culminating in the Chv\'atal--Szemer\'edi theorem \cite{MR0609100}, has been focused on determining the optimal ``sufficiently large'' condition in this theorem.

A closely related line of research was initiated by Nikiforov \cite{MR2409174,MR2398823}, who proved the following remarkable theorem. 
\begin{theorem}[Nikiforov \cite{MR2409174,MR2398823}]\label{thm:nikiforov}
    Let $H$ be an $h$-vertex graph, let $\gamma>0$, and let $n$ be sufficiently large. If $G$ is an $n$-vertex graph with at least $\gamma n^h$ copies of $H$, then $G$ contains an $H$-blowup $H[k]$, for some
    \[
        k \geq c_H(\gamma) \log n,
    \]
    where $c_H(\gamma)>0$ is a constant depending only on $\gamma$ and $H$.
\end{theorem}
This result gives the best possible dependence on $n$, since a standard computation shows that a random $n$-vertex graph has, with high probability, $\Omega(n^h)$ copies of $H$ and no copy of $H[k]$ for any $k \geq 2\log n$. We remark too that one natural approach to proving such a theorem---passing to an auxiliary $h$-uniform hypergraph whose edges are the $H$-copies in $G$---does not work, and can only prove a bound of $k = \Omega((\log n)^{\frac{1}{h-1}})$. Thus, \cref{thm:nikiforov} is one of many theorems (in addition to the famous (6,3) theorem \cite{MR0519318}, among others) capturing the idea that the $H$-copies in a graph $G$ have extra structure beyond what is found in a general $h$-uniform hypergraph.

While \cref{thm:nikiforov} gives the optimal $n$-dependence, the optimal dependence on $\gamma$ and $H$ is unknown. Nikiforov proved that, for any fixed graph $H$, one can take $c_{H}(\gamma) = \Omega(\gamma^{h})$ if $H=K_h$ is a complete graph, and $c_H(\gamma)=\Omega(\gamma^{h^2})$ in general. However, this is very far from the best known upper bound $c_H(\gamma) = O_H(1/{\log \frac 1 \gamma})$, which again comes from considering a random graph of the appropriate edge density, namely $\gamma^{1/e(H)}$.

\cref{thm:nikiforov} is an extremely useful result with many applications (e.g.\ \cite{MR2520282,MR2993136,MR4625871,MR4195582}), and as such, there have been several attempts to improve the bounds on $c_H(\gamma)$. R\"odl and Schacht \cite{MR2993136} proved that we may take $c_{K_h}(\gamma) \geq \gamma^{1+o(1)}$ when $H$ is complete, and Fox--Luo--Wigderson \cite{MR4195582} improved this to $c_H(\gamma) \geq \gamma^{1-1/e(H)+o(1)}$ for all $H$. However, just as in Nikiforov's original argument, all of these bounds are still polynomial in $\gamma$, whereas the best known upper bound is logarithmic. The only case where the truth is known is when $H$ is bipartite (which is a degenerate case of the problem); in this case, the K\H ov\'ari--S\'os--Tur\'an theorem \cite{MR0065617} immediately implies that $c_H(\gamma) = \Omega_H(1/{\log \frac 1\gamma})$, matching the upper bound up to a constant factor.

Our main result proves the same bound for all triangle-free graphs $H$, yielding the first non-trivial case where the optimal bound in \cref{thm:nikiforov} is known.
\begin{theorem}\label{thm:main}
    For every triangle-free graph $H$ on $h$ vertices, there exists a constant $\alpha_H>0$ such that the following holds for all $0<\gamma\leq \frac 12$ and all $n$. If $G$ is an $n$-vertex graph with at least $\gamma n^h$ copies of $H$, then $H[k] \subseteq G$, for some
    \[
        k \geq \alpha_H \frac{\log n}{\log \frac 1 \gamma}.
    \]
\end{theorem}
It is natural to conjecture that the same result is true for all graphs $H$. However, it appears that proving this, even in the simplest case of $H=K_3$, would require substantial new techniques.

As a consequence of \cref{thm:main}, we obtain a surprising result about multicolor Ramsey numbers. Recall that, for a graph $F$ and an integer $q \geq 2$, the \emph{Ramsey number} $r(F;q)$ is defined as the least integer $N$ such that every $q$-coloring of $E(K_N)$ contains a monochromatic copy of $F$. In general, our understanding of $r(F;q)$ is rather poor; for example, it is a major open problem \cite{MR1815606,MR4357431,MR0177846} to determine whether $r(K_3;q)$ grows exponentially or super-exponentially with $q$, and an even more major open problem \cite{MR366726,MR2552114,MR4548417,2303.09521} to determine the growth rate of $r(K_k;2)$ as $k \to \infty$. However, for complete bipartite graphs, our understanding is fairly complete, and it is known \cite{MR0360329} that
\begin{equation}\label{eq:multicolor K_kk}
    q^{ck} \leq r(K_{k,k};q) \leq q^{Ck}
\end{equation}
for all $q,k \geq 2$, where $C>c>0$ are absolute constants. Here, the lower bound follows from a random coloring, and the upper bound follows from the K\H ov\'ari--S\'os--Tur\'an theorem. In particular, this implies that for fixed $F$,
$r(F;q)$ grows polynomially in $q$ if $F$ is bipartite, whereas it is easy to see\footnote{Indeed, since $K_{2^q}$ is the edge-union of $q$ bipartite graphs, we have $r(F;q)>2^q$ in case $F$ is non-bipartite.}  that $r(F;q)$ grows at least exponentially in $q$ if not. However, our next result shows that if $F$ is a large blowup of a fixed triangle-free graph, then the dependence on $q$ \emph{does} eventually become polynomial.
\begin{theorem}\label{thm:multicolor ramsey}
    Let $H$ be an $h$-vertex triangle-free graph, and let $q \geq 2$ be an integer. If $k \geq 100^h q^{4h^2}$, then
    \[
        r(H[k];q) \leq q^{\Lambda_H k},
    \]
    where $\Lambda_H>0$ is a constant depending only on $H$.
\end{theorem}
This result is best possible up to the constant $\Lambda_H$, since a random coloring again witnesses that $r(H[k];q)\geq q^{ck}$ for every non-empty graph $H$ and some absolute constant $c>0$. Moreover, as discussed above, the assumption that $k$ is sufficiently large is also necessary, since the dependence on $q$ is super-polynomial if $k$ is fixed and $q$ is large whenever $H$ is non-bipartite. In fact, this argument shows that the ``sufficiently large'' condition on $k$ in \cref{thm:multicolor ramsey} is nearly best possible in terms of the $q$-dependence, in that $k$ must be at least of order $q^{\Omega_H(1)}$ for such a statement such as \cref{thm:multicolor ramsey} to be true.

\subsection{Discussion and corollaries}
One interesting feature of \cref{thm:main} is that we do not assume that $n$ is sufficiently large with respect to $\gamma$, in contrast to previous results on this topic. In particular, \cref{thm:main} gives a non-trivial result even when $\gamma$ is a small negative power of $n$, as stated in the following result.
\begin{corollary}\label{cor:generalized turan}
    Let $H$ be an $h$-vertex triangle-free graph, and let $G$ be an $n$-vertex graph. If $G$ contains at least $n^{h-\alpha_H/k}$ copies of $H$, then $H[k] \subseteq G$.
\end{corollary}
Indeed, \cref{cor:generalized turan} follows immediately from \cref{thm:main} by plugging in $\gamma = n^{-\alpha_H/k}$. This result can be equivalently stated in the language of \emph{generalized extremal numbers} \cite{MR3548290}, where we recall that $\ex(n, H, F)$ denotes the maximum number of copies of $H$ that can appear in an $n$-vertex $F$-free graph. In this language, \cref{cor:generalized turan} states that $$\ex(n,H,H[k]) < n^{h-\alpha_H/k}$$ for all triangle-free $H$. For general graphs $H$, the best known upper bound for this problem follows by a reduction to a hypergraph extremal problem, which yields the bound
\begin{equation}\label{eq:generalized turan}
\ex(n,H,H[k]) = O_H\left(n^{h - 1/k^{h-1}}\right)
\end{equation}
for any $h$-vertex graph $H$. Recently, several authors \cite{2402.16818,2405.07763,MR4546048} have attempted to improve this bound; in particular, the results of \cite{2402.16818,2405.07763} imply that $\ex(n,H,H[k]) = o(n^{h-1/k^{h-1}})$. However, their proof techniques rely on the (hyper)graph removal lemma, and therefore give only a very slight improvement over \eqref{eq:generalized turan}. In contrast, \cref{cor:generalized turan} gives a power-savings improvement over \eqref{eq:generalized turan} whenever $k$ is sufficiently large in terms of $H$; to the best of our knowledge, this is the first example of such a power-savings improvement for a non-bipartite graph $H$. We remark that proving an analogue of \cref{cor:generalized turan} in the first non-trivial case of $H=K_3$ would, in particular, make progress towards a conjecture of Fox, Sankar, Simkin, Tidor, and Zhou \cite[Conjecture 6.4]{2401.00359} on the extremal numbers of Latin squares.

More generally, one can let $\gamma=\gamma(n)$ decay to $0$ at some rate as $n \to \infty$, and \cref{thm:main} can still yield useful results. Results along these lines for other choices of $\gamma(n)$ have implications for certain hypergraph extremal problems. For example, \cref{thm:main} shows that if $G$ has at least $\exp(-\sqrt{\log n})n^h$ copies of a triangle-free graph $H$, then it contains a blowup of $H$ of order $\Omega(\sqrt{\log n})$. R\"odl and Schacht \cite[Problem 3]{MR2993136} showed that such a statement for $H=K_3$ would yield a result like \cref{thm:nikiforov} in hypergraphs, which remains a major open problem. Similarly, Conlon, Fox, and Sudakov \cite[Theorem 1.1]{MR2959395} proved a certain analogue of the Erd\H os--Hajnal conjecture for $3$-uniform hypergraphs, but conjectured \citetext{\citealp{{MR2959395}}, Conjecture 1; \citealp{MR3497267}, Conjecture 3.16} that their result could be quantitatively strengthened. The main barrier to improving their result is strengthening a key technical lemma \cite[Lemma 3.3]{MR2959395}, which would roughly boil down to proving that an $n$-vertex graph with $\Omega(n^3/{\log n})$ triangles contains a copy of $K_3[k]$ with $k \geq (\log n)^{1-o(1)}$. Again, \cref{thm:main} implies that such a result holds if we replace $K_3$ by any triangle-free graph.

While blowups are interesting in their own right, results about graph blowups are generally of great utility, and often immediately imply results about other graphs. For example, the following is an immediate consequence of \cref{thm:multicolor ramsey}, combined with a deep result of \L uczak \cite{MR2260851} on the \emph{homomorphism threshold} of triangle-free graphs.
\begin{corollary}\label{cor:homomorphism threshold}
    For every $\beta>0$, there exists $C_\beta>0$ such that the following holds for all $q \geq 2$ and all sufficiently large $k$. If $F$ is a $k$-vertex triangle-free graph with minimum degree at least $(\frac 13 + \beta)k$, then
    \[
        r(F;q) \leq q^{C_\beta k}.
    \]
\end{corollary}
Indeed, the result of \L uczak \cite{MR2260851} mentioned above states that there exists a triangle-free graph\footnote{In \cite[Corollary 4.3(3)]{brandt-thomasse}, a precise description of these graphs $H_\beta$ is given; they are the so-called \emph{Vega graphs}.} $H_\beta$, depending only on $\beta$, such that any graph $F$ satisfying the assumption of \cref{cor:homomorphism threshold} is a subgraph of $H_\beta[k]$, and thus \cref{cor:homomorphism threshold} is an immediate consequence of \cref{thm:multicolor ramsey}.

As discussed above, it is natural to conjecture that \cref{thm:main} holds for all $H$. If true, this would imply that \cref{thm:multicolor ramsey} holds for all $H$; in particular, the $H=K_h$ case of this result would imply that $r(F;q)$ is polynomial in $q$ whenever $F$ is a sufficiently large graph of bounded chromatic number. We believe that such a statement is interesting in its own right, as it shows that the bipartite behavior carries through whenever $\chi(F)$ is bounded. Moreover, proving such a statement may be easier than extending \cref{thm:main} to all $H$. Even the $H=K_3$ case seems interesting and challenging.
\begin{conjecture}
    For all $q \geq 2$ and all sufficiently large $k$, we have
    \[
        r(K_{k,k,k};q) \leq q^{Ck},
    \]
    where $C>0$ is an absolute constant.
\end{conjecture}

\subsection{Organization}
We provide high-level proof sketches of our main results in \cref{sec:outlines}.
We then proceed to prove \cref{thm:main} in \cref{sec:nikiforov}, and \cref{thm:multicolor ramsey} in \cref{sec:ramsey}. We end in \cref{sec:conclusion} with some concluding remarks.

All logarithms in this paper are to base $2$. We systematically omit floor and ceiling signs whenever they are not crucial.

\section{Proof outlines}\label{sec:outlines}
Although the proofs of \cref{thm:main,thm:multicolor ramsey} are fairly short, we now give an outline of their proofs, beginning with \cref{thm:main}. 
\subsection{Proof sketch of Theorem \ref{thm:main}}
Let us fix a triangle-free graph $H$ with vertex set $\{u_1,\dots,u_h\}$, and an $n$-vertex graph $G$ with at least $\gamma n^h$ copies of $H$. By a standard averaging argument, we can pass to a partite setting, namely finding disjoint sets $V_1,\dots,V_h \subseteq V(G)$ containing at least $\gamma \prod_{i=1}^h \ab{V_i}$ \emph{canonical} copies of $H$, where a copy is canonical if the $i$th vertex of $H$ lies in $V_i$ for all $i \in [h]$.

Without loss of generality, let us assume that the neighbors of vertex $u_h$ in $H$ are vertices $u_1,u_2,\dots,u_t$, for some $t<h$. We construct an auxiliary bipartite graph $\Gamma = (A,B,E)$, where $A = V_1\times \dots\times V_t$ and $B = V_{t+1}\times\dots \times V_h$, in which $(v_1,\dots,v_t)\sim_\Gamma (v_{t+1},\dots,v_h)$ if $(v_1,\dots,v_h)$ form a canonical copy of $H$. The crucial property of $\Gamma$, and the only place in the proof where we use the triangle-freeness of $H$, is the following: for every $(v_{t+1},\dots,v_h) \in B$, its neighborhood in $\Gamma$ is of the form $S_1 \times \dots \times S_t \subseteq A$, for some sets $S_i \subseteq V_i$. In other words, neighborhoods of vertices in $B$ are not arbitrary subsets of $A$, but rather highly structured product sets, for the following reason. For a given vertex $(v_{t+1},\dots,v_h) \in B$, if its components do not form a copy of $H[\{u_{t+1},\dots,u_h\}]$, then it has no neighbors in $A$. In the other case, by definition, a vertex $(v_1,\dots,v_t) \in A$ is a neighbor of $(v_{t+1},\dots,v_h) \in B$ if and only if, for every $i$, the vertex $v_i$ is adjacent in $G$ to all $v_j$ with $u_iu_j \in E(H)$. However, since $H$ is triangle-free, the set $\{u_1,\dots,u_t\} \subseteq V(H)$ is an independent set, hence for any $i\in [t]$ the set of such $j$ is contained in $\{t+1,\dots,h\}$. In other words, the condition that $(v_1,\dots,v_t)$ is a neighbor of $(v_{t+1},\dots,v_h)$ is simply the conjunction of $t$ different conditions, one for each $i \in [t]$, and there is no interaction between these conditions. This proves that the neighborhood of $(v_{t+1},\dots,v_h)$ does indeed have the claimed product structure. As an immediate consequence, we note that the common neighborhood of any $B' \subseteq B$ also has a product structure, since the intersection of product sets is another product set.

Now, we note that since $G$ contains many canonical copies of $H$, the graph $\Gamma$ must be dense; in fact, it has at least $\gamma \ab A \ab B$ edges. We now apply the dependent random choice technique to this graph\footnote{We refer to the survey \cite{MR2768884} for an introduction to this powerful technique. For this high-level proof overview, we assume some familiarity with the technique, and defer a detailed discussion to the full proof of \cref{thm:main}.}, by choosing a set $B^*\subseteq B$ of $s = \Theta(\log n/{\log (1/\gamma)})$ random vertices from $B$. One can check that with positive probability, the following three ``good events'' occur:
\begin{itemize}
    \item letting $A'\subseteq A$ be the common neighborhood of $B^*$, we have that $\ab{A'} \geq \ab A/{\sqrt n}$;
    \item there are at least $\gamma^C \ab{A'} \ab B$ edges of $\Gamma$ between $A'$ and $B$, for some constant $C>0$;
    \item and the vertices in $B^*$, which are elements of $V_{t+1}\times \dots \times V_h$, have at least $s/2$ distinct final coordinates.
\end{itemize}
Unpacking what this all means, and recalling the product structure of $A'$, we see that we can find the following structure. There is a set $V_h^* \subseteq V_h$, of size at least $s/2 = \Omega(\log n/{\log(1/\gamma)})$, and sets $V_1' \subseteq V_1,\dots,V_t' \subseteq V_t$, each of size at least $\sqrt n$, such that $V_h^*$ is complete to each $V_i'$, and the number of canonical copies of $H' \coloneqq H\setminus \{u_h\}$ among $V_1',\dots,V_t',V_{t+1},\dots,V_{h-1}$ is at least $\gamma^C \ab{V_1'}\dotsb \ab{V_t'}\ab{V_{t+1}}\dotsb\ab{V_{h-1}}$. If we can find a blowup $H'[k]$, for $k = \Omega(\log n/{\log(1/\gamma)})$ among these $h-1$ parts, we can combine it with $V_h^*$ to obtain the desired blowup $H[k]$. And crucially, we are now in a position to proceed by induction: although we have lost a great deal in the number of copies (from $\gamma$ to $\gamma^C$) and have lost many vertices (e.g.\ from $\ab{V_1}=n$ to $\ab{V_1'}=\sqrt n$), these losses are only \emph{polynomial}. Since we are aiming for a bound of the form $\Omega_H(\log n/{\log(1/\gamma)})$, these polynomial losses will be converted by the logarithms to linear losses, which can be absorbed into the constant factor. 

\subsection{Proof sketch of Theorem \ref{thm:multicolor ramsey}}
Thanks to \cref{thm:main}, in order to prove \cref{thm:multicolor ramsey}, it suffices to show that every $q$-coloring of $E(K_N)$, for an appropriately chosen $N$, contains many monochromatic copies of $H$ in some color. 
A simple averaging argument, due to Erd\H os \cite{MR0151956}, shows that such a statement is true, namely that for every $H,q$, there exists some $c_{H,q}>0$ such that every $q$-coloring of $E(K_N)$ contains at least $(c_{H,q}-o(1))N^h$ monochromatic copies of $H$, where $H$ has $h$ vertices and the $o(1)$ term tends to $0$ as $N \to \infty$. The quantity $c_{H,q}$ is called the \emph{$q$-color Ramsey multiplicity constant} of $H$; for more on this topic see \cite[Section 2.6]{MR3497267}, as well as \cite{2206.05800,MR4618272,MR4426765} for more recent developments. 

Unfortunately, this result is too weak quantitatively to directly plug into \cref{thm:main}. Indeed, one can show\footnote{We can $q$-color $E(K_N)$ by blowing up a $(q-1)$-coloring on $2^{q-1}$ vertices in which every color class is bipartite, and using the $q$th color for all edges inside a part of the blowup. As $H$ is connected and non-bipartite, the only monochromatic copies of $H$ in this coloring lie inside one of these parts, showing that $c_{H,q} \leq 2^{-(q-1)(h-1)}$.} that 
if $H$ is connected and non-bipartite, then $c_{H,q} \leq 2^{-(q-1)(h-1)} =2^{-\Omega_H(q)}$. Therefore, if we simply plug this into \cref{thm:main}, the best bound we could hope to prove is of the form $r(H[k];q) \leq 2^{O_H(qk)}$.

So rather than directly apply this Ramsey multiplicity argument, we first pass to an appropriate sub-configuration of the coloring. Namely, we show that in every $q$-coloring of $E(K_N)$, there is a set\footnote{In the formal proof, it is more convenient to work in a partite setting as above, so we actaully prove the existence of large sets $V_1,\dots,V_h$ containing many canonical copies of $H$.} $V \subseteq V(K_N)$ with $\ab V = \Omega_H(N)$, which contains $q^{-O_H(1)} \ab V^h$ monochromatic copies of $H$ in some color, say red. In other words, by slightly shrinking the vertex set, we are able to boost the Ramsey multiplicity from exponentially to polynomially small in $q$. At this point, we can apply \cref{thm:main} to the red graph on vertex set $V$ and find a red copy of $H[k]$.

Proving the existence of such a set $V$ is fairly straightforward using Szemer\'edi's regularity lemma (see \cref{sec:ramsey}). However, we give an alternative proof, based on a direct density-increment argument, that is more quantitatively efficient; in particular, it is this extra efficiency that allows us to assume that $k$ is only polynomially large with respect to $q$ in \cref{thm:multicolor ramsey}.

\section{Proof of Theorem \ref{thm:main}}\label{sec:nikiforov}
\begin{definition}
    Let $G$ be a graph, let $H$ be a graph with vertex set $\{u_1,\dots,u_h\}$, and let $V_1,\dots,V_h$ be disjoint subsets of $V(G)$. We say that the tuple $(V_1,\dots,V_h)$ is a \emph{$\gamma$-inflation of $H$} if $(V_1,\dots,V_h)$ contains at least $\gamma \prod_{i=1}^h \ab{V_i}$ canonical copies of $H$.

    In case $\ab{V_i}=n$ for all $i$, we say that this tuple is a $(\gamma,n)$-inflation of $H$. If, instead, we only have that $\ab{V_i} \geq n$ for all $i$, we say it is a $(\gamma,{\geq n})$-inflation of $H$.
\end{definition}

We prove \cref{thm:main} by induction on $h$; the following lemma gives the inductive step, and \cref{thm:main} follows from it fairly straightforwardly. We denote by $N_H(u),d_H(u)$ the neighborhood and degree, respectively, of a vertex $u \in V(H)$; we omit the $H$ subscripts when the graph is clear from context. Also, given a graph $G$ and two disjoint vertex subsets $U,W \subseteq V(G)$, we denote by $G[U,W]$ the bipartite subgraph of $G$ induced on the pair $(U,W)$. 

\begin{lemma}\label{lem:induct}
    Let $H$ be graph with vertex set $\{u_1,\dots,u_h\}$, and suppose that there are no triangles in $H$ containing vertex $u_h$ (or equivalently, that $N_H(u_h)$ is an independent set). There exists a constant $C_H=\max \{1,8d_H(u_h)\}\geq 1$ such that the following holds for all $0<\gamma \leq \frac 12$ and $n \geq \gamma^{-16h}$.
    
    Let $G$ be a graph, and let $V_1,\dots,V_h \subseteq V(G)$ form a $(\gamma,\geq n)$-inflation of $H$.
    Then there exist $V_h^*\subseteq V_h$ and $V_i'\subseteq V_i$ for $i\in [h-1]$ so that:
    \begin{enumerate}
        \item $|V_h^*|\ge \frac{\log(n)}{8\log(1/\gamma)}$;
        \item $G[V_j',V_h^*]$ is a complete bipartite graph for all $j$ such that $u_j\in N_H(u_h)$;
        \item $(V_1',\dots,V_{h-1}')$ forms a $(\gamma',\geq n')$-inflation of $H \setminus \{u_h\}$, where $n' =\sqrt n$ and $\gamma' = \gamma^{C_H}$.
    \end{enumerate}
\end{lemma}
We remark that the choice of $n' = \sqrt n$ is not particularly important; any bound of the form $n' = n^{\Omega_H(1)}$ would suffice to prove \cref{thm:main}.

Before proving \cref{lem:induct}, we first show how it implies \cref{thm:main}. We actually state and prove a partite version of this result; \cref{thm:main} follows from this partite version by a standard random partitioning argument, which we include for completeness after the proof of \cref{thm:nikiforov partite}.
\begin{theorem}\label{thm:nikiforov partite}
    Let $H$ be a triangle-free graph. There exists a constant $\beta_H>0$ such that the following holds for all $0<\gamma\leq \frac 12$ and all $n$. 

    Let $G$ be a graph, and let $V_1,\dots,V_h \subseteq V(G)$ form a $(\gamma,\geq n)$-inflation of $H$. Then $H[k] \subseteq G$, where $k \geq \beta_H \frac{\log(n)}{\log(1/\gamma)}$. In fact, there exist $W_i \subseteq V_i$, where $\ab{W_i}=k$, which form a copy of $H[k]$.
\end{theorem}

\begin{proof}
    We prove the theorem by induction on $h\coloneqq\ab{V(H)}$, with the base case $h=1$ vacuously true.

    Inductively, suppose we have proved the statement for all triangle-free graphs $H'$ on $h-1$ vertices, and fix a triangle-free graph $H$ on vertex set $\{u_1,\dots,u_h\}$. Let $H' = H \setminus \{u_h\}$ and
    \[
    \beta_H = \min \left\{ \frac{1}{16h},  \frac{\beta_{H'}}{2C_H}\right\},
    \]
    where $C_H$ is the constant from \cref{lem:induct}.
    We claim that this value of $\beta_H$ suffices. 

    Suppose first that $n < \gamma^{-16h}$. In this case we have that
    \[
    \beta_H \frac{\log(n)}{\log(1/\gamma)} < \beta_H \frac{16h\log(1/\gamma)}{\log (1/\gamma)} =16h\beta_H \leq 1,
    \]
    and thus the statement is trivial since we may pick an arbitrary canonical copy of $H$, and set $W_i$ to be a singleton containing its $i$th vertex. Therefore, we may assume henceforth that $n \geq \gamma^{-16h}$. Let $n' = \sqrt n$ and $\gamma' = \gamma^{C_H}$.
    
    Let $V_1,\dots,V_h$ form a $(\gamma,\geq n)$-inflation of $H$. By \cref{lem:induct}, we may find $V_h^* \subseteq V_h$ and $V_i' \subseteq V_i$ for $i \in [h-1]$ such that $\ab{V_h^*} \geq  \frac{\log(n)}{8\log(1/\gamma)}$, $G[V_j', V_h^*]$ is complete for all $u_j \in N_H(u_h)$, and  $(V_1',\dots,V_{h-1}')$ forms a $(\gamma',\geq n')$-inflation of $H'$. By the inductive hypothesis, we may now find $W_i' \subseteq V_i'$, for all $i \in [h-1]$, which form a copy of $H'[k']$, where
    \[
    k' \geq \beta_{H'} \frac{\log(n')}{\log(1/\gamma')} = \beta_{H'} \frac{\log(\sqrt n)}{\log(1/\gamma^{C_H})} = \frac{\beta_{H'}}{2C_H} \cdot \frac{\log(n)}{\log(1/\gamma)} \geq \beta_H \frac{\log(n)}{\log(1/\gamma)}=k,
    \]
    where the final inequality holds by our definition of $\beta_H$. Moreover, we have that
    \[
    \ab{V_h^*} \geq  \frac{\log(n)}{8\log(1/\gamma)} \geq \beta_H \frac{\log(n)}{\log(1/\gamma)}=k,
    \]
    again by the choice of $\beta_H$. Therefore, if we pick arbitrary subsets $W_1 \subseteq W_1', \dots,W_{h-1} \subseteq W_{h-1}', W_h \subseteq V_h^*$, each of order $k$, we have found the claimed canonical copy of $H[k]$. This completes the proof of the inductive step.
\end{proof}
Tracing through the quantitative dependencies in this proof, it is straightforward to verify inductively that this proof demonstrates $\beta_H \geq (100 h)^{-h}$. In fact, if $H$ is $d$-degenerate, then by applying the induction according to the degenerate ordering, one can improve this bound to $\beta_H \geq (100d)^{-h}$.
\begin{remark}
    Using the same proof, it is not hard to show the following strengthening of \cref{thm:nikiforov partite}. Fix an independent set $I \subseteq V(H)$, and suppose we are again in the setting of \cref{thm:nikiforov partite}. Then we may again find $W_i \subseteq V_i$ forming a blowup of $H$, where $\ab{W_i} \geq \beta_H \frac{\log(n)}{\log(1/\gamma)}$ for all $i$ with $u_i \notin I$, and $\ab{W_i}\geq n^{\sigma_H}$ for all $i$ with $u_i \in I$, for some $\sigma_H>0$. That is, we may ensure that in this blowup, the sets corresponding to vertices in $I$ have polynomial, rather than logarithmic, size. In fact, by appropriately modifying \cref{lem:induct}, one can even take $\sigma_H$ arbitrarily close to $1$, at the expense of obtaining a worse bound on $\beta_H$. We leave the details to the interested reader.
\end{remark}

For completeness, we include the simple derivation of \cref{thm:main} from \cref{thm:nikiforov partite}.
\begin{proof}[Proof of \cref{thm:main}]
    Let $\alpha_H = \min \{\beta_H/h, 1/(2\log h)\}$, where $\beta_H$ is the constant from \cref{thm:nikiforov partite}. Let $G$ be an $n$-vertex graph with at least $\gamma n^h$ copies of $H$.

    First note that if $n <h^2$, we have that
    \[
    \alpha_H \frac{\log(n)}{\log(1/\gamma)} < \alpha_H \log(h^2)= 2\alpha_H \log(h) \leq 1,
    \]
    where we use the assumption $\gamma \leq \frac 12$ and our choice of $\alpha_H \leq 1/(2\log h)$. In this case the result is trivial, since it suffices to return a copy of $H=H[1]$ in $G$, which exists by assumption. We may therefore assume henceforth that $n \geq h^2$.
    
    We randomly partition $V(G)$ into $h$ equally-sized parts $V_1,\dots,V_h$, where the partition is chosen uniformly at random among all equitable partitions. Note that each labeled copy of $H$ in $G$ becomes a canonical copy among $V_1,\dots,V_h$ with probability at least $h^{-h}$. Therefore, by linearity of expectation and the assumption that $G$ contains at least $\gamma n^h$ copies of $H$ (and in particular at least $\gamma n^h$ \emph{labeled} copies of $H$), we conclude that there exists a choice of $V_1,\dots,V_h$ such that the number of canonical copies of $H$ is at least
    \[
        \frac{\gamma n^h}{h^h} = \gamma \left( \frac nh \right)^h=\gamma \prod_{i=1}^h \ab{V_i}.
    \]
    We fix such a choice of $V_1,\dots,V_h$. This is a $(\gamma,n/h)$-inflation of $H$ in $G$. By \cref{thm:nikiforov partite}, we conclude that $G$ contains a copy of $H[k]$, where
    \[
    k \geq \beta_H \frac{\log(n/h)}{\log(1/\gamma)} \geq \beta_H \frac{\log(n)/h}{\log(1/\gamma)} \geq \alpha_H \frac{\log(n)}{\log(1/\gamma)},
    \]
    where the second inequality holds since $n \geq h^2$.
\end{proof}
The rest of this section is dedicated to proving \cref{lem:induct}. 
We begin with a simple subsampling fact.
\begin{lemma}\label{lem:subsample}
    Let $(V_1,\dots,V_h)$ be a $(\gamma, \geq n)$-inflation of $H$. There exist $V_1'\subseteq V_1,\dots,V_h'\subseteq V_h$ so that $(V_1',\dots,V_h')$ is a $(\gamma,n)$-inflation of $H$.
\end{lemma}
\begin{proof}
    Let $V_i'$ be a random subset of $V_i$, chosen uniformly at random among all subsets of size exactly $n$, and make these choices independently for all $i \in [h]$. Each canonical copy of $H$ survives in $(V_1',\dots,V_h')$ with probability $\prod_{i=1}^h \frac{\ab{V_i'}}{\ab{V_i}}$. Therefore, by linearity of expectation, there exists a choice with at least $\gamma \prod_{i=1}^h \ab{V_i'}$ canonical copies, as claimed.
\end{proof}

We are now ready to prove \cref{lem:induct}.

\begin{proof}[Proof of \cref{lem:induct}]
Note that by \cref{lem:subsample}, we may assume that $(V_1,\dots,V_h)$ forms a $(\gamma,n)$-inflation of $H$.
Let $t$ be the degree of vertex $u_h$ in $H$, and assume without loss of generality that $N_H(u_h) = \{u_1,u_2,\dots,u_t\}$. 
We set $C_H = \max \left\{ 1, 8t \right\}$. 
There is nothing to prove if $t=0$, since in this case we may set $V_h^*=V_h$ and $V_1'=V_1,\dots,V_{h-1}'=V_{h-1}$. Therefore, we assume henceforth that $t\geq 1$, and therefore that $C_H=8t$.

We create an auxiliary bipartite graph $\Gamma = (A,B,E)$, where $A = V_1\times \dots\times V_t$ and $B = V_{t+1}\times\dots \times V_h$, where $(v_1,\dots,v_t)\sim_\Gamma (v_{t+1},\dots,v_h)$ if $(v_1,\dots,v_h)$ form a canonical copy of $H$. Clearly $|A|= n^t, |B|=n^{h-t}$, and $|E|\ge \gamma n^h = \gamma |A||B|$. For a vertex $b \in B$ and an index $t +1 \leq j \leq h$, we write $(b)_j$ to denote the $j$ coordinate of $b$.

Let $s = \cei{\frac{\log(n)}{4\log(1/\gamma)}}$. Note that our assumption $n \geq \gamma^{-16h}$ is equivalent to saying that $\frac{\log (n)}{\log(1/\gamma)}\geq 16h$. In particular, this implies that
\begin{equation}\label{eq:s LB}
s \geq \frac{\log(\gamma^{-16h})}{4\log(1/\gamma)} = 4h \geq 4t.
\end{equation}
It also implies that $\frac{\log(n)}{3 \log(1/\gamma)} - \frac{\log(n)}{4 \log(1/\gamma)}>1$, and thus that
$\frac{\log(n)}{4 \log(1/\gamma)}\leq s \leq \frac{\log(n)}{3 \log(1/\gamma)}$. Hence, we conclude that $n^{-1/3}\leq \gamma^s \leq n^{-1/4}$.
We now sample $b_1,\dots,b_s\in B$ uniformly at random (with repetitions), and let $A'\subseteq A$ denote their common neighborhood in $\Gamma$. Note that $A'$ is empty if, for some $\ell \in [s]$, the components of $b_\ell$ do not form a copy of $H[\{u_{t+1},\dots,u_h\}]$. 
Crucially, if this does not happen, then the fact that $N_H(u_h)$ is an independent set in $H$ implies that $A'$ has a product structure, namely that $A' = S_1 \times \dots \times S_t$, where
\[
S_i =
V_i\cap \bigcap_{\substack{\ell \in [s]\\u_j \in N_H(u_i)}} N_G((b_\ell)_j),
\]
Indeed, since $N_H(u_h)$ is an independent set, a tuple $(v_1,\dots,v_t)$ lies in $A'$ if and only if each $v_i$ is a common neighbor (in $G$) of all vertices $(b_\ell)_j$, where $\ell \in [s]$ and $u_j \sim_H u_i$. This is the only place in the proof where we use the assumption that $u_h$ lies in no triangles in $H$.

Write $X\coloneqq\ab{A'}$. By Jensen's inequality, we have that 
\[
\E[X] = \sum_{a\in A}\left(\frac{d_\Gamma(a)}{\ab B}\right)^s \ge \gamma^s \ab A = \gamma^s n^t\ge n^{t-1/3},
\]
where in the final inequality we recall that $\gamma^s \geq n^{-1/3}$.

Let $\gamma'=\gamma^{C_H}$, and let $Y$ count the number of $a\in A'$ with $d_\Gamma(a)\le \gamma' n^{h-t}$. For such a vertex, the probability that it is included in $A'$ is at most $(\gamma')^s$, and thus 
\[\E[Y]\le n^t(\gamma')^{s}=n^t (\gamma^s)^{8t} \leq n^t n^{-2t} = n^{-t},\]
where we use that $\gamma^s \leq n^{-1/4}$.

Finally, let $Z$ be the indicator random variable for the event that $\ab{\{(b_1)_h,\dots,(b_s)_h\}}\leq s/2$, that is, that at most $s/2$ vertices in $V_h$ are used as the final coordinate of one of $b_1,\dots,b_s$. We have that 
\[
\E[Z] \leq \binom{n}{s/2} \left(\frac{s/2}{n}\right)^s \leq \left(\frac{en}{s/2}\right)^{s/2} \left(\frac{s/2}{n}\right)^s = \left(\frac{es}{2n}\right)^{s/2} \leq n^{-s/4} \leq n^{-t},
\]
where we use that $es/2 \leq 2s \leq \log n \leq \sqrt n$ in the penultimate inequality\footnote{This is the only place where we use that $\gamma \leq \frac 12$, to say that $2s \leq \log n$.}, and \eqref{eq:s LB} in the final inequality.

Thus, we have that $\E[X-n^t(Y+Z)]\geq n^{t-1/3}-2$. Thus we can fix an outcome of $b_1,\dots,b_s$ and $A'$ where this occurs. Since $X\le n^t$ and $Y,Z$ take non-negative integer values, we must have $Y=Z =0$ and $X\ge n^{t-1/3}-2 \geq n^{t-1/2}$. Thus, recalling that $A' = S_1\times \dotsb \times S_t$, we have $\min_i\ab{S_i} \ge |A'|/n^{t-1} \ge \sqrt{n}=n'$. 

Next, since $Y=0$, we have that $$e(A',B)\ge \gamma'|A'||B| = \gamma' \prod_{i=1}^t\ab{S_i} \prod_{j=t+1}^{h}\ab{V_j}.$$ 
Thus, we find that $(S_1,\dots,S_t,V_{t+1},\dots,V_h)$ form a $(\gamma',\geq n')$-inflation of $H$. Let $V_i' = S_i$ for $1 \leq i \leq t$ and $V_i' = V_i$ for $t+1 \leq i \leq h-1$.
Since every canonical copy of $H'$ in $V_1' \cup \dots \cup V_{h-1}'$ extends to at most $\ab{V_h}$ canonical copies of $H$, we conclude that $(V_1',\dots,V_{h-1}')$ form a $(\gamma', \geq n')$-inflation of $H'$.

Finally, we set $V_h^*\coloneqq  \{(b_1)_h,\dots,(b_s)_h\}$. Since $Z =0$, we have that  
\[
\ab{V_h^*} \geq \frac{s}{2} \geq \frac{\log(n)}{8\log(1/\gamma)}.
\]
Meanwhile, by definition of $A'$, we have that $V_i'= S_i \subseteq \bigcap_{v\in V_h^*}N_G(v)$, for all $i \in [t]$. This is the same as saying that $G[V_i', V_h^*]$ is complete, which concludes the proof.    
\end{proof}

\section{Proof of Theorem \ref{thm:multicolor ramsey}}\label{sec:ramsey}

\cref{thm:multicolor ramsey} follows from \cref{thm:main} together with the following lemma.
\begin{lemma}\label{lem:monochromatic inflation}
    Let $q,h \geq 2$ and $N\geq q^{100^h q^{4h^2}}$ be integers, and let $\eta=q^{-100^h q^{4h^2}}$. In every $q$-coloring of $E(K_N)$, there exists a $((2q)^{-e(K_h)}, \geq \eta N)$-inflation of $K_h$ in some color $i \in [q]$.
\end{lemma}

We begin by showing how \cref{thm:multicolor ramsey} follows from \cref{thm:main,lem:monochromatic inflation}.
\begin{proof}[Proof of \cref{thm:multicolor ramsey}]
Fix an $h$-vertex triangle-free graph $H$ and an integer $q \geq 2$, and let $\gamma = q^{-h^2}$, $\eta = q^{-100^h q^{4h^2}}$, and $\Lambda_H=2h^2/\beta_H$, where $\beta_H$ is the constant from \cref{thm:nikiforov partite}. 
Fix an integer $k \geq 100^h q^{4h^2}$, and let $N = q^{\Lambda_H k}$ and $n = \eta N$. Note that $N \geq 1/\eta^2$ by our lower bound assumption on $k$ and since $\Lambda_H \geq 2$, hence $n \geq \sqrt N$.

We claim that $r(H[k];q) \leq N$. Indeed,
fix a $q$-coloring of $E(K_N)$. Note that by our lower bound on $k$, we have that $N \geq q^{100^h q^{4h^2}}$, so we may apply \cref{lem:monochromatic inflation} to find a $((2q)^{-e(K_h)}, \geq \eta N)$-inflation of $K_h$ in some color, say red. By our choices of $\gamma$ and $n$, this is a $(\gamma, \geq n)$-inflation of $K_h$, and hence also a $(\gamma, \geq n)$-inflation of $H$. We now apply \cref{thm:nikiforov partite} to conclude that the red graph contains a blowup of $H$ whose parts have size at least
\[
\beta_H \frac{\log(n)}{\log(1/\gamma)} \geq \beta_H\frac{\log(\sqrt N)}{h^2 \log(q)} =\frac{\beta_H \Lambda_H k}{2h^2} = k.
\]
That is, we have found a red copy of $H[k]$, completing the proof.
\end{proof}

The rest of this section is dedicated to proving \cref{lem:monochromatic inflation}. We remark that for the qualitative statement---namely that any $q$-coloring of $E(K_N)$ contains a $((2q)^{-e(K_h)}, \geq\eta N)$-inflation of $K_h$ for \emph{some} $\eta>0$---there is a fairly simple proof using standard arguments from regularity theory. Before proceeding with the proof of \cref{lem:monochromatic inflation}, we sketch this alternative proof\footnote{For this high-level proof sketch, we assume some familiarity with the basic concepts related to Szemer\'edi's regularity lemma. None of this will be used in the sequel, so a reader unfamiliar with these topics should feel free to skip to the next paragraph.}.

Let $\varepsilon = \varepsilon(H,q)>0$ be sufficiently small. Fix a $q$-coloring of $E(K_N)$. By applying a colored version of Szemer\'edi's regularity lemma (see e.g.\ \cite[Theorem 1.18]{MR1395865}), we may partition $V(K_N)$ into a bounded number of parts, such that each of the color classes is $\varepsilon$-regular with respect to this partition. By Tur\'an's theorem and by picking $\varepsilon$ sufficiently small, we may pass to $r(K_h;q)$ parts, such that all pairs between them are $\varepsilon$-regular in each of the colors. We now define an auxiliary coloring of $K_{r(K_h;q)}$ by coloring an edge according to the most popular color in the corresponding pair. By the definition of $r(K_h;q)$, there is a monochromatic copy of $K_h$ in this coloring, meaning that we can find $h$ parts $V_1,\dots,V_h$ such that all pairs between them are $\varepsilon$-regular and have edge density at least $1/q$ in some color, say red. The $K_h$ counting lemma implies that $V_1,\dots,V_h$ contain at least $(q^{e(K_h)}-\delta(\varepsilon))\prod_{i=1}^h \ab{V_i}$ canonical red copies of $K_h$, where $\delta(\varepsilon)$ tends to $0$ as $\varepsilon \to 0$. By choosing $\varepsilon$ sufficiently small, we conclude that $(V_1,\dots,V_h)$ form the desired inflation of $K_h$ in red. Due to the use of the regularity lemma, this proof gives terrible, tower-type bounds on $\eta$. One can obtain better bounds, of the form $2^{-2^{(qh)^{O(1)}}}$, by replacing the use of the regularity lemma by the cylinder regularity lemma \cite{MR1333857} (see also \cite{MR4170446} for a colored version of the cylinder regularity lemma).

However, we shall present an alternative proof which avoids such regularity techniques, and gives the stronger bound on $\eta$ claimed in \cref{lem:monochromatic inflation}. Thanks to this improved bound, we have that \cref{thm:multicolor ramsey} holds when $k$ is a polynomial in $q$.

\subsection{Preliminary lemmas}
We begin by collecting a few lemmas that will be used in the proof of \cref{lem:monochromatic inflation}.
 Our first lemma is standard, and says that a density decrement in one part of a bipartite graph can be converted to a density increment in a complementary part.
 We denote by $d_{A}(y) \coloneqq \ab{N(y) \cap A}$ the degree of a vertex $y$ into some vertex set $A$.

\begin{lemma}\label{lem:other increment}
    Let $\Gamma = (X,Y,E)$ be a bipartite graph where $d(y)\ge p|X|$ for all $y\in Y$. Suppose that $Y'\subseteq Y$ and $X^*\subseteq X$ are such that $d_{X^*}(y)\le (1-\xi)p|X^*|$ for each $y\in Y'$. 
    
    Then defining $X' \coloneqq X\setminus X^*$ and writing $c \coloneqq |X^*|/|X|$, we have that $d_{X'}(y)\ge (1+c\xi)p|X'|$ for all $y\in Y'$.  
\end{lemma}
\begin{proof}For any $y \in Y'$, we have
\[
d_{X'}(y) = d(y)-d_{X^*}(y) \ge p|X|-(1-\xi)p|X^*|= p|X'|+\xi p c|X| \ge (1+c\xi)p|X'|.\qedhere
\]
\end{proof}
Our next lemma shows that every dense graph contains a large, bipartite subgraph satisfying a one-sided minimum degree condition.
\begin{lemma}\label{lem:dense to bipartite}
    Let $n \geq 2$, let $G$ be an $n$-vertex graph with at least $\rho \binom n2$ edges, and let $\eps\in (0,1]$. Then there exist disjoint $A,B \subseteq V(G)$ with $\ab A, \ab B \geq (\eps^2/32)\rho n$ such that $d_A(y) \geq (1-\eps)\rho \ab A$ for all $y \in B$.
\end{lemma}
\begin{proof}
    Let $V=V(G)$.
    Let $V_{\mathrm{light}}\subseteq V$ be the set of vertices with $d(v)\le (1-\eps/2)\rho (n-1)$, and let $V_{\mathrm{heavy}}\coloneqq V\setminus V_{\mathrm{light}}$. We have that 
    \[
    2e(G) - \sum_{v \in V_{\mathrm{light}}} d(v) = \sum_{v \in V_{\mathrm{heavy}}} d(v) \leq \ab{V_{\mathrm{heavy}}} (n-1)
    \]
    and that
    \[
     2e(G)-\sum_{v\in V_{\mathrm{light}}}d(v) \geq 2\rho \binom n2 - \ab{V_{\mathrm{light}}}\left(\left(1-\frac \eps 2\right) \rho (n-1)\right) \geq \frac \eps 2 \rho n(n-1).
    \]
    Combining these two bounds, we find that $\ab{V_{\mathrm{heavy}}}\geq (\eps \rho/2)n$.

    There is nothing to prove if $\rho=0$, as we may then take $A=B=\emptyset$, so we may assume $\rho>0$.
    If $n\leq 32/(\rho\eps^2)$, then we are again done, as we may take $A,B$ to both be singletons, namely the endpoints of some edge, which exists since $e(G) \geq \rho\binom n2>0$. We thus assume henceforth that $(\eps/4)n\geq 1$.
    We let $B_0\subseteq V$ be a uniformly random set of size $\flo{ (\eps/4)n} $. 
    We claim that for every $v \in V_{\mathrm{heavy}}$, we have that
    \begin{equation}\label{eq:heavy vtx}
        \P\left(d_{V\setminus B_0}(v)< (1-\eps)\rho (n-1) ~\middle|~ v\in B_0\right) \leq \frac 12.
    \end{equation}
    Indeed, fix some $v \in V_{\mathrm{heavy}}$, and let $D=d(v)$. The event that $d_{V \setminus B_0}(v) < (1-\varepsilon)\rho(n-1)$ is precisely the event that among the $D$ neighbors of $v$, more than $D-(1-\varepsilon)\rho(n-1)$ of them are included in $B_0$. Therefore,
    \begin{align*}
       \P\left(d_{V\setminus B_0}(v)< (1-\eps)\rho (n-1) ~\middle|~ v\in B_0\right) &= \P\left(\ab{(B_0 \setminus \{v\}) \cap N(v)}>D-(1-\varepsilon)\rho(n-1) ~\middle|~ v \in B_0\right)\\
        &\leq \frac{\E[\ab{(B_0 \setminus \{v\})\cap N(v)} \mid v \in B_0]}{D-(1-\varepsilon)\rho(n-1)},
    \end{align*}
    by Markov's inequality. 
    Conditional on the event $v \in B_0$, the distribution of $B_0 \setminus\{v\}$ is that of a uniformly random subset of $V \setminus \{v\}$ of size exactly $\flo{(\varepsilon/4)n}-1$.
    As $\flo{(\eps/4)n}-1 \leq (\eps/4)(n-1)$, a given vertex of $V \setminus \{v\}$ is included in $B_0 \setminus \{v\}$ with probability at most $\eps/4$, and therefore
    \[
    \E[\ab{(B_0 \setminus \{v\})\cap N(v)} \mid v \in B_0] \leq \frac \eps 4 \ab{N(v)} = \frac{\eps D}{4},
    \]
    by linearity of expectation and the assumption $\ab{N(v)}=D$. Continuing the computation above, we conclude that
    \begin{align*}
       \P\left(d_{V\setminus B_0}(v)< (1-\eps)\rho (n-1) ~\middle|~ v\in B_0\right)
        &\leq \frac{\E[\ab{(B_0 \setminus \{v\})\cap N(v)} \mid v \in B_0]}{D-(1-\varepsilon)\rho(n-1)}\\
        &\leq \frac{\eps D/4}{D-(1-\eps)\rho(n-1)}.
    \end{align*}
    Recall that $D \geq (1-\eps/2)\rho(n-1)$, since $v \in V_{\mathrm{heavy}}$.
    This final expression is a decreasing function of $D$, hence it is maximized when $D=(1-\eps/2)\rho(n-1)$. Plugging this in, we find that
    \begin{align*}
       \P\left(d_{V\setminus B_0}(v)< (1-\eps)\rho (n-1) ~\middle|~ v\in B_0\right)
        &\leq \frac{\eps D/4}{D-(1-\eps)\rho(n-1)}\\
        &\leq \frac{(\eps/4)(1-\eps/2)\rho(n-1)}{(1-\eps/2)\rho(n-1)-(1-\eps)\rho(n-1)}\\
        &=\frac{(\eps/4)(1-\eps/2)\rho(n-1)}{(\eps/2)\rho(n-1)}\\
        &\leq \frac 12,
    \end{align*}
    proving \eqref{eq:heavy vtx}.
    
    Now, define $B$ to comprise all vertices $v \in B_0 \cap V_{\mathrm{heavy}}$ with $d_{V \setminus B_0}(v) \geq (1-\eps)\rho (n-1)$. 
    From \eqref{eq:heavy vtx}, we find that
    \[
    \E[\ab B] = \sum_{v \in V_{\mathrm{heavy}}} \P(v \in B_0) \P(v \in B \mid v \in B_0) \geq \frac 12 \sum_{v \in V_{\mathrm{heavy}}} \P(v \in B_0) \geq \frac \eps {16} \ab{V_{\mathrm{heavy}}} \geq \frac{\eps^2\rho}{32}n,
    \]
    using the fact that each vertex is included in $B_0$ with probability at least $\varepsilon/8$ (since $\flo{(\varepsilon/4)n}\geq1$) and that $\ab{V_{\mathrm{heavy}}}\geq (\varepsilon\rho/2)n$
    Fix an outcome where $\ab B \geq (\eps^2/32)\rho n$, and let $A = V \setminus B_0$. Note that $\ab A = n - \flo{(\eps/4)n}\geq (\eps^2/32)\rho n$, as claimed. Finally, by the definition of $B$, we have that
    \[
    d_A(y) \geq (1-\eps)\rho (n-1) \geq (1-\eps)\rho \ab A
    \]
    for all $y \in B$, where we use the fact that $\ab A \leq n-1$.
\end{proof}

Although our goal in \cref{lem:monochromatic inflation} is only to find an inflation of $K_h$ in some color, it will be convenient to inductively maintain a richer structure, which we now define.
\begin{definition}
    Let $V_1,\dots,V_h$ be pairwise disjoint vertex subsets of some graph $G$.
    We say $(V_1,\dots,V_h)$ is a $(\rho,\eps)$-rich inflation of $K_h$ if $e(G[V_1])\ge \rho \binom{|V_1|}{2}$ and $(V_1,\dots,V_h)$ is a $((1-\eps)\rho)^{e(K_h)}$-inflation of $K_h$. If additionally $\ab{V_i}\geq n$ for all $i$, we say it is a $(\rho,\eps,\geq n)$-rich inflation of $K_h$.
\end{definition}
The next lemma is the key step in our density increment strategy. It shows that at every step, we may either improve a rich inflation of $K_h$ to a rich inflation of $K_{h+1}$, or perform a density-increment step.

\begin{lemma}\label{lem: rich increment}
    Suppose that $(V_1,\dots,V_h)$ is a $(\rho,\eps,\geq n)$-rich inflation of $K_h$ for some $h \geq 1$ and some $\eps \in (0,\frac 19]$ and $\rho \in (0,\frac 12]$, and let $Y = \bigcup_{i=1}^h V_i$. Let $X$ be some set of vertices, disjoint from $Y$, with  $d_X(y)\ge p|X|$ for all $y\in Y$, where $p \geq (1-\eps)\rho$. Then at least one of the following holds.
    \begin{itemize}
        \item The tuple $(V_1,\dots,V_h,X)$ is a $(\rho,2\eps,\geq \min\{n,\ab X\})$-rich inflation of $K_{h+1}$, or
        \item there exist $X' \subseteq X, Y' \subseteq Y$ with $|X'|\ge \eps \rho^{h^2}|X|,|Y'|\ge \eps \rho^{h^2} n$ which satisfy $d_{X'}(y)\ge (1+\eps \rho^{3h})p|X'|$ for all $y\in Y'$.
    \end{itemize}
\end{lemma}
The proof of \cref{lem: rich increment} is fairly technical, but the main idea is simple, so we explain it before proceeding with the formal proof. We pick uniformly random vertices $y_1 \in V_1,\dots,y_h \in V_h$, where we think that we make these choices sequentially. Throughout, the key random variable that we track is the size $\ab{N_i}$ of the common neighborhood of $y_1,\dots,y_i$ in $X$, and we track the evolution of this random variable as $i$ increases from $1$ to $h$. 

By the assumption that $(V_1,\dots,V_h)$ is a rich inflation of $K_h$, we have a good probability that $y_1,\dots,y_h$ form a copy of $K_h$; let us henceforth condition on this event. If, with good probability, $\ab{N_h}$ is large at the end of the process, then this means that there are many vertices in $X$ which complete $(y_1,\dots,y_h)$ to a copy of $K_{h+1}$; in this case, we have the first outcome of \cref{lem: rich increment}, namely that $(V_1,\dots,V_h,X)$ is a rich inflation of $K_{h+1}$. Therefore, we may assume that with good probability, $\ab{N_h}$ is small at the end of the process. Since we started with $\ab{N_0}=\ab X$, there must be some step $i$ at which, with good probability, $\ab{N_i}$ shrinks significantly relative to $\ab{N_{i-1}}$. But this precisely means that there are many choices of $y_i \in V_i$ which have few neighbors in $N_{i-1}$, i.e.\ we have found large subsets $Y' \subseteq V_i\subseteq Y$ and $X^* \subseteq N_{i-1}\subseteq X$ with few edges between them. Applying \cref{lem:other increment} allows us conclude that in this case, we have the second outcome of \cref{lem: rich increment}.

\begin{proof}[Proof of \cref{lem: rich increment}]

    We begin by handling the case $h=1$. In this case, the fact that $d_X(y) \geq p \ab X$ for all $y \in Y=V_1$ implies that the number of canonical copies of $K_2$ in $(V_1,X)$ is at least $p \ab X \ab Y \geq (1-\eps)\rho \ab X \ab {V_1}$. Hence the tuple $(V_1,X)$ a $(\rho,\eps,\geq n)$-rich inflation of $K_2$, and therefore a $(\rho,2\eps,\geq n)$-rich inflation as well. This is one of the two claimed outcomes, and thus we may assume henceforth that $h \geq 2$.

By deleting arbitrary edges, we may assume that $d_X(y) = p|X|$ for all $y \in Y$. Indeed, both claimed outcomes are monotone under adding edges, so if we prove the result after these edge deletions we have also proved it for the original graph. 
We suppose henceforth that $(V_1,\dots,V_h,X)$ is not a $(\rho,2\eps)$-rich inflation of $K_h$, and seek to prove the existence of the claimed sets $X',Y'$.
    
Recall that we assumed that $(V_1,\dots,V_h)$ is a $(\rho,\eps)$-rich inflation, implying that $e(G[V_1]) \geq \rho \binom{\ab{V_1}}2$; hence the fact that $(V_1,\dots,V_h,X)$ is not a $(\rho,2\eps)$-rich inflation implies that the number of canonical $K_{h+1}$ is less than $((1-2\eps)\rho)^{e(K_{h+1})} \prod_{i=1}^h \ab{V_i}\cdot \ab X$.

Let $(y_1,\dots,y_h)\in V_1\times \dots \times V_h$ be chosen uniformly at random, and let $N_h \subseteq X$ be their common neighborhood in $X$.
Let $\EE$ denote the event that $y_1,\dots,y_h$ form a copy of $K_h$, and note that $\P(\EE) \geq ((1-\eps)\rho)^{e(K_h)}$ by assumption. 
Let $x \in X$ be chosen uniformly at random, and let $\EE'$ be the event that $(y_1,\dots,y_h,x)$ form a copy of $K_{h+1}$; again by assumption, we have that $\P(\EE')<((1-2\eps)\rho)^{e(K_{h+1})}$. Therefore,
\[
((1-2\eps)\rho)^h = \frac{((1-2\eps)\rho)^{e(K_{h+1})}}{((1-2\eps)\rho)^{e(K_h)}} > \frac{\P(\EE')}{\P(\EE)} = \P(\EE' \mid \EE).
\]
Moreover, note that conditional on the event $\EE$, the probability of $\EE'$ is exactly $\ab{N_h}/\ab X$. Therefore, we conclude that
\[
\E[\ab{N_h} \mid \EE] \leq ((1-2\eps)\rho)^{h}\ab X.
\]
Noting that $(1-2\eps)\rho \leq (1-\eps)^2\rho \leq (1-\eps)p$, where the final step is by our assumption on $p$, we find that $\E[\ab{N_h} \mid \EE] \leq ((1-\eps)p)^h \ab X$. This in turn implies that
\[
((1-\eps)p)^h \ab X \geq \E[\ab{N_h} \mid \EE] \geq (1-\eps)p^h \ab X\cdot \P(\ab{N_h} \geq (1-\eps)p^h \ab X\mid \EE),
\]
and thus $\P(\ab{N_h} \geq (1-\eps)p^h \ab X \mid \EE) \leq (1-\eps)^{h-1} \leq 1-\eps$. Therefore, $\P(\ab{N_h} \leq (1-\eps)p^h \ab X \mid \EE) \geq \eps$. Denoting by $\delta \coloneqq \P(\EE)$, we conclude that
$\P(|N_h|\leq (1-\eps)p^h|X|)\ge \eps \delta$. 

We note now that if $p > 1-1/(2h)$, then $\ab{N_h} > \frac 12 \ab X$ with probability $1$, by the pigeonhole principle. On the other hand, in this case we also have $(1-\eps)p^h > (1-\eps)(1-1/(2h))^h \geq (9/16)(1-\eps)\geq\frac 12$, where we use the assumptions $h \geq 2$ and $\eps \leq 1/9$. Thus, in this case we have $\P(\ab{N_h} \leq (1-\eps) \rho^h \ab X) = 0$, a contradiction. Therefore we may assume henceforth that $p \leq 1-1/(2h)$.

Let $N_0\coloneqq X$ and for $i =1,\dots ,t$ set $N_i\coloneqq N_{i-1}\cap N(y_i)$. Note that this agrees with our earlier definition of $N_h$. We now observe that if $\ab{N_\tau} \geq (1-\eps/(2h)) p \ab{N_{\tau-1}}$ for all $\tau \in [h]$, then we have
\[
\frac{\ab{N_h}}{\ab X} = \prod_{\tau=1}^h \frac{\ab{N_\tau}}{\ab{N_{\tau-1}}} \geq \left( \left( 1- \frac{\eps}{2h}\right) p\right)^h > (1-\eps) p^h.
\]
In other words, whenever the event $\ab{N_h}\leq (1-\eps)p^h \ab X$ occurs, there 
must exist some minimal index $\tau\in [t]$ where $|N_\tau|<(1-\eps/(2h))p|N_{\tau-1}|$. If this event does not occur, we set $\tau\coloneqq \infty$.

By the pigeonhole principle, there must be some index $i^*\in [h]$ so that $\P(\tau = i^*)\ge \eps\delta/h$. Whence, there is some choice of $y_1^*,\dots,y_{i^*-1}^*$ so that $\P(\tau = i^*|y_1=y_1^*,\dots,y_{i^*-1}=y_{i^*-1}^*) \ge \eps\delta/h $. Writing $X^*\coloneqq  \bigcap_{i<i^*}N(y_i^*)$, and $Y' \coloneqq  \{y\in V_{i^*}: d_{X^*}(y)<(1-\eps/(2h))p|X^*|\}$, we get that $|Y'| \ge (\eps\delta/h)|V_{i^*}|\ge (\eps \delta/h)n$.

Additionally, we have that $|X^*|\ge ((1-\eps/(2h))p)^{i^*-1}|X|\ge \frac 12 p^{h-1}|X|$, where the final inequality uses that $\eps \leq 1$ and that $(1-1/(2h))^h \geq \frac 12$ for all $h$. At the same time, by our minimum degree assumption, we cannot have $i^*=1$, whence $|X^*|\le d_X(y_1^*) = p|X|$. We now apply \cref{lem:other increment} with $\xi = \eps/(2h)$ and $c \geq \frac 12 p^{h-1}$ to conclude that for $X' = X \setminus X^*$, we have $d_{X'}(y) \geq (1+c\xi)p \ab{X'}$ for all $y \in Y'$. The claimed result then follows by noting that
\[
\ab {X'} = \ab X - \ab{X^*} \geq (1-p)\ab X \geq \frac 1{2h} \ab X \geq \eps \rho^{h^2} \ab X,
\]
that
\[
\ab{Y'} \geq \frac{\eps \delta}{h}n \geq \frac{\eps((1-\eps)\rho)^{e(K_h)}}{h}n \geq \eps \rho^{h^2}n,
\]
and that
\[
c \xi \geq \frac{\eps p^{h-1}}{4h} \geq \frac{\eps((1-\eps)\rho)^{h-1}}{4h} \geq \eps \rho^{3h}.\qedhere
\]
\end{proof}

\subsection{The density increment argument}
\begin{definition}
    Let $q,h \geq 2$ be integers and $\eps \in (0,1]$ a parameter. We define
    $\eta(q,h,\eps)$ to be the maximum $\eta \in [0,1]$ so that for all $N\geq 1/\eta$, and for every $q$-coloring of $E(K_N)$, there exists a monochromatic $(1/q,\eps,\ge \eta N)$-rich inflation of $K_h$.
\end{definition}
We note that the set of such $\eta \subseteq [0,1]$ is a closed subset of $[0,1]$, hence we really can define $\eta(q,h,\eps)$ as a maximum, rather than a supremum. The next result gives a recursive lower bound on $\eta(q,h+1,\eps)$ in terms of $\eta(q,h,\eps/2)$.
\begin{proposition}\label{prop:main increment}
    For every $q\geq 2,h \geq 1$ and $\eps \in (0,\frac 19]$, we have
    \[
    \eta(q,h+1,\eps) \geq \left(\eps^4 q^{-h^2} \eta(q,h,\eps/2)\right)^{q^{6h}/\eps}.
    \]
\end{proposition}

\begin{proof}

Fix some $q,h,\eps$ and consider $N\ge 1$ and some coloring $\chi:E(K_N)\to [q]$. For a color $i\in [q]$ we let $G_i\subseteq K_N$ denote the graph of edges receiving color $i$. We set $\rho \coloneqq 1/q$, $\eta_0\coloneqq \eta(q,h,\eps/2)$, and $k\coloneqq q^{6h}/\eps$. Note that this choice of $k$ implies that
\begin{equation}\label{eq:k choice}
    (1+\eps \rho^{3h})^{k/q} \geq \exp \left( \frac{\eps \rho^{3h}k}{2q}\right) \geq \exp \left( \frac{\eps k}{q^{3h+2}}\right)\geq e^{q} > 2q \geq \frac{1}{(1-\eps)\rho},
\end{equation}
where we use the inequality $1+x \geq e^{x/2}$ (valid for all $x \in [0,1]$) in the first step, the definition of $\rho$ and the assumption $q \geq 2$ in the second step, the definition of $k$ in the third step, and the definition of $\rho$ and the assumption $\eps \leq \frac 12$ in the final step.

We run the following process over several rounds.
At time $t$, we will have a set of vertices $Y= Y^{(t)} \subseteq V(K_N)$, a set of colors $S= S^{(t)}\subseteq [q]$, and for $i\in S$ a subset $X_i= X^{(t)}_i\subseteq V\setminus Y$. We shall further maintain that for all $i\in S$ and $y\in Y$, that $|N_{G_i}(y)\cap X_i| \ge p_i|X_i|$ for some certain $p_i= p_i^{(t)}$, which we will ensure satisfies $p_i\up t\ge (1-\eps)\rho$.

We shall initialize with $Y^{(0)} = V(K_N)$ and $S^{(0)} = \emptyset$ (whence there are no $X_i$'s or $p_i$'s to define). Now, in a step, we either have $|Y|\le 1/\eta_0$ (in which case we halt), or we can find some $V_1,\dots,V_{h}\subseteq Y$ which form a $(\rho,\eps/2,\ge \eta_0 |Y|)$-rich inflation of $K_h$ in some color $i\in [q]$. 

If $i \not \in S$, we set $S^{(t+1)} \coloneqq  S^{(t)}\cup \{i\}$. Then, using the fact that $e(G_i[V_1]) \geq \rho \binom{\ab{V_1}}2$, we may apply \cref{lem:dense to bipartite}. That lemma outputs a pair of sets $A,B \subseteq Y$, and we define $X_i \up{t+1}\coloneqq A, Y\up{t+1} \coloneqq B$ to be this pair, satisfying $\ab{X_i\up{t+1}}, \ab{Y\up{t+1}} \geq (\eps^2 \rho/32)\ab{V_1}\geq (\varepsilon^2\rho/32)(\eta_0\ab Y)$. Note that, since we shrink $Y= Y\up t$ to a subset $Y\up{t+1}$, we maintain the claimed properties for all colors in $S\up t$. For the new color $i$, which we have added to $S\up{t+1}$, we have by \cref{lem:dense to bipartite} that every vertex in $Y\up{t+1}$ has at least $(1-\eps)\rho \ab{X\up{t+1}}$ $G_i$-neighbors in $X\up{t+1}$, hence we also maintain the claimed property for the new color $i$, with $p_i\up {t+1} \geq (1-\eps)\rho$.

Otherwise, $i\in S$. Now if we have that adding $X_i$ to $V_1,\dots,V_{h}$ creates a $(\rho,\eps)$-rich inflation of $K_{h+1}$ in color $i$, we halt the process. If this does not happen, we can apply \cref{lem: rich increment} to pass to an increment $Y'\subseteq Y,X'\subseteq X_i$ where $|Y'|\ge \eps \rho^{h^2}(\eta_0 \ab{Y})$ and $|X'|\ge \eps \rho^{h^2}|X_i|$ where the density boosts by a factor of $(1+\eps\rho^{2h})$. We now set $X_i \up{t+1}=X'$ and $Y\up{t+1} = Y'$, and maintain $S\up{t+1}=S\up t$. By \cref{lem: rich increment}, we have $p_i\up{t+1} \geq (1+\eps \rho^{2h})p_i\up t$.

We claim that the process can only run for at most $k$ rounds. Indeed, if it runs longer, then there are $k'\coloneqq (k/q)+1$ steps $t_1<\dots <t_{k'}$ where we pick the same color $i$. At time $t_1$, we have $p_i\up{t_1}\ge (1-\eps)\rho$. This implies that $p_i\up {t_{k'}} \ge (1+\eps\rho^{2h})^{k'-1}p_i\up {t_1}>1$, by \eqref{eq:k choice}. This is impossible, showing that the process indeed halts after at most $k$ rounds.

Note that at every step of this process, we have
\[
\frac{\ab{Y\up{t+1}}}{\ab{Y\up t}} \geq \eps^4 \rho^{h^2} \eta_0,
\]
since $Y\up {t+1}$ is obtained by shrinking $Y$ by a factor of $\eps^2 \rho\eta_0/32$ or $\eps \rho^{h^2}\eta_0$, both of which are lower-bounded by $\eps^4 \rho^{h^2}\eta_0$ since $\eps \leq \frac 18$ and $\rho \leq 1$. Similarly, when we introduce a set $X_i$ it has size at least $(\eps^2 \rho/32)\eta_0 \ab Y$, and every subsequent step shrinks it by at most $\eps^4 \rho^{h^2} \eta_0$.

In other words, if we start with $N \geq (\eps^4 \rho^{h^2}\eta_0)^{-k}$, then we will be able to keep this process going without ever halting because we shrink too much. Thus, we can only halt by outputting a $(\rho,\eps)$-rich inflation of $K_{h+1}$. Moreover, for the same reason, when we output a $(\rho,\eps)$-rich inflation of $K_{h+1}$, its parts all have size at least $(\eps^4 \rho^{h^2}\eta_0)^k N$, which implies that
\[
\eta(q,h+1,\eps) \geq (\eps^4 \rho^{h^2} \eta_0)^k = (\eps^4 q^{-h^2} \eta(q,h,\eps/2))^{q^{6h}/\eps}.\qedhere
\]
\end{proof}
As a corollary of \cref{prop:main increment}, we obtain the following bound for $\eta(q,h,\eps)$.
\begin{corollary}\label{cor:eta LB}
    For every $q\geq 2,h \geq 1$ and $\eps \in (0,\frac 19]$, we have
    \[
    \eta(q,h,\eps) \geq q^{-q^{4h^2}/\eps^{2h}}
    \]
\end{corollary}
\begin{proof}
    We prove this by induction on $h$. The base case $h=1$ is immediate: by the pigeonhole principle, one of the $q$ color classes in any $q$-coloring of $E(K_N)$ must contain at least $\frac 1q \binom N2$ edges, which is precisely a rich inflation of $K_1$ in this color. This shows that $\eta(q,1,\eps) = 1$ for all $\eps>0$, and in particular implies the claimed bound for $h=1$.

    Inductively, suppose we have proved the claimed result for some $h\geq 1$. By \cref{prop:main increment}, we have
    \begin{equation*}
        \eta(q,h+1,\eps) \geq \left(\eps^4 q^{-h^2} \eta(q,h,\tfrac\eps2)\right)^{q^{6h}/\eps} \geq \left( \eps^4 q^{-h^2} q^{-q^{4h^2}/(\eps/2)^{2h}}\right)^{q^{6h}/\eps} \geq \left( q^{-(4/\eps +h^2 + q^{4h^2}/(\eps/2)^{2h})}\right)^{q^{6h}/\eps},
    \end{equation*}
    where the final step uses that
     $\eps^4 \geq 2^{-4/\eps} \geq q^{-4/\eps}$ since $x \geq 2^{-1/x}$ for all $x \in (0,1]$. Therefore, if we denote by $\zeta \coloneqq -\log_q \eta(q,h+1,\eps)$ the negative of the exponent above, we conclude that
     \begin{align*}
         \zeta \leq \frac{q^{6h}}{\eps}\left( \frac 4 \eps + h^2 + \frac{q^{4h^2}}{(\eps/2)^{2h}}\right) =\frac{q^{6h}}{\eps}\left( \frac 4 \eps + h^2 + 2^{2h}\cdot \frac{q^{4h^2}}{\eps^{2h}}\right).
     \end{align*}
     Note that $4/\eps \leq q^{4h^2}/\eps^{2h}$, and similarly $h^2 \leq q^{4h^2}/\eps^{2h}$. Therefore,
     \[
     \zeta \leq \frac{q^{6h}}{\eps} \left( (2^{2h}+2) \frac{q^{4h^2}}{\eps^{2h}}\right) \leq 2^{2h+1} \frac{q^{4h^2+6h}}{\eps^{2h+1}} \leq \frac{q^{4h^2+8h+1}}{\eps^{2h+2}} \leq \frac{q^{4(h+1)^2}}{\eps^{2(h+1)}}.
     \]
     Recalling the definition of $\zeta$, this implies that
     \[
     \eta(q,h+1,\eps) \geq q^{-q^{4(h+1)^2}/\eps^{2(h+1)}},
     \]
     completing the inductive step.
\end{proof}

We are now finally ready to prove \cref{lem:monochromatic inflation}.
\begin{proof}[Proof of \cref{lem:monochromatic inflation}]
By \cref{cor:eta LB}, we have that $\eta(q,h,1/10) \geq q^{-100^h q^{4h^2}} = \eta$, where we recall the definition of $\eta$ from the statement of \cref{lem:monochromatic inflation}. Additionally, our assumption on $N$ implies that $N \geq 1/\eta \geq 1/\eta(q,h,1/10)$. Therefore, by the definition of $\eta(q,h,\eps)$, we conclude that every $q$-coloring of $E(K_N)$ contains a monochromatic $(1/q, 1/10, \geq \eta N)$-rich inflation of $K_h$. In particular, this is a $(\gamma,\geq \eta N)$-inflation of $K_h$, where
\[
\gamma = \left(\left(1-\frac 1{10}\right) \frac 1 q\right)^{e(K_h)} \geq (2q)^{-e(K_h)},
\]
which is precisely what we wanted to prove.
\end{proof}

\section{Concluding remarks}\label{sec:conclusion}
We remark that, although we stated \cref{thm:nikiforov partite} under the assumption $\gamma \leq \frac 12$, essentially the same result actually holds for all $\gamma \in (0,1)$, including when $\gamma = 1-o(1)$. 
\begin{proposition}\label{prop:gamma to 1}
    If $G$ contains a $(\gamma,n)$-inflation of $H$, then $H[k]\subseteq G$, where
    \[
    k = \Omega_H \left( \frac{\log((1-\gamma)n)}{\log(1/\gamma)}\right).
    \]
\end{proposition}
Note that this bound is again best possible, and generalizes \cref{thm:nikiforov partite}; in particular, the extra factor of $1-\gamma$ in the numerator is immaterial when $\gamma \leq \frac 12$ as in \cref{thm:nikiforov partite}, but is necessary and sharp when $\gamma \to 1$ (indeed, already for $H=K_2$, a random bipartite graph of edge density $1-n^{-(1-c)}$ contains no $K_{k,k}$ with $k=10c\log(n)n^{1-c}$).

To prove Proposition~\ref{prop:gamma to 1}, one argues as follows.
Recall that for small $\eps$, $\log(1/(1-\eps)) = \Theta(\eps)$, so writing $\gamma = 1-\eps$ we now wish to prove that $k = \Omega_H( \frac{\log(\eps n)}{\eps})$ for $\eps<1/2$. 
To do so, we shall establish a more precise version of Lemma~\ref{lem:induct}, where we now have $|V_h^*| = \Omega_H(\frac{\log(\eps n)}{\log(1/\gamma)}), n'= \frac{(\eps n)^{1/3}}{\eps}$ and $\gamma' = \gamma^{C_H}\ge 1-C_H\eps$. Once this is established, the result follows by induction on $h$ as in the proof of Theorem~\ref{thm:nikiforov partite} (the important point is that we still have $\log((1-\gamma')n') = \Omega_H(\log((1-\gamma)n))$, so passing to this subgraph is still `cheap' --- this would not be the case if $n' = \sqrt{n}<1/\eps$).

The refined version of Lemma~\ref{lem:induct} can be split into two regimes. First, we note that if $\gamma \le 1-\Omega(n^{-1/2})$, then the proof from Section~\ref{sec:nikiforov} works essentially unchanged (giving $|V_h^*| \ge \Omega(\frac{\log(n)}{\log(1/\gamma)}),n' = \sqrt{n}$). Indeed, the only place where we used the assumption $\gamma \leq \frac 12$ is in estimating $\E[Z]$ in the proof of \cref{lem:induct}, and a similar estimate (which is sufficiently strong for the rest of the proof) holds for all $\gamma \leq 1-\Omega(n^{-1/2})$. In particular, if $\gamma \le 1-n^{-1/4}$, we have that $n' =\sqrt{n} \ge \Omega(\frac{(\eps n)^{1/3}}{\eps})$, so we are done if $\gamma \le 1-n^{-1/4}$

So now we may assume $\gamma = 1-\eps $ for $ \eps \le n^{-1/4}$. Here, given a $(\gamma,n)$-inflation $(V_1,\dots,V_h)$ of $H$, we begin by deleting all vertices in any $V_i$ with fewer than $(1-h^2\eps)n$ neighbors in any other $V_j$ with $ij \in E(H)$. Since any such vertex lies in at most $(1-h^2\eps) n^{h-1}$ canonical copies of $H$, after deleting them we still have a $(1-O_H(\eps), \geq n/2)$-inflation of $H$. We now apply the same dependent random choice argument as in the proof of \cref{lem:induct}, with two changes: we now sample $s = \Theta(\log(\eps n)/\eps)$ random vertices, chosen so that $\gamma^s n = \Theta((\eps n)^{1/3}/\eps)$, and we no longer keep track of the random variable $Y$. The rest of the argument goes through unchanged, giving us the set $V_h^*$ of size $\ab{V_h^*}\geq s/2$ and subsets $V_1',\dots,V_t'$, each of size at least $\gamma^s n = (\eps n)^{\Omega(1)}/\eps$. Moreover, although we no longer have the random variable $Y$ to guarantee that we preserve many canonical copies of $H' = H \setminus \{h\}$, the minimum degree condition we established does guarantee this. Noting that
\[
\ab{V_h^*} \geq \frac s2 = \Omega_H \left(\frac{\log(\eps n)}{\eps}\right) = \Omega_H \left( \frac{\log((1-\gamma)n)}{\log(1/\gamma)}\right),
\]
we see that we get the desired result.

While our work gives optimal results for triangle-free graphs, a number of important problems remain open,
the most significant of which is to prove results like \cref{thm:main,thm:multicolor ramsey} for other graphs $H$; in particular, as discussed in the introduction, such results for $H=K_3$ would have far-reaching consequences.

A more modest question, which we nonetheless find interesting, is to obtain the  correct bounds for \cref{cor:eta LB}. It would be nice to show that $\eta(q,h, q^{-C})\geq q^{-O_h(q)}$, or possibly even $\eta(q, H, q^{-C})\geq r(H;q)^{-O_H(1)}$ for general graphs $H$. Such bounds would give essentially optimal bounds on how large $k$ needs to be for \cref{thm:multicolor ramsey} to hold. In the case of $h=3$, one can use recent work of Kelley, Lovett, and Meka \cite{2308.12451} in order to do the density-increment argument more efficiently and prove that $\eta(q,3,q^{-C})\geq q^{-O(q\log^2(q))}$. However, we already hit a stumbling block at $h=4$, and do not know how to establish $\eta(q,4,q^{-C})\geq q^{-q^{1+o(1)}}$.

\section*{Acknowledgments} We are grateful to the anonymous referee for many helpful comments.

\bibliographystyle{amsplain}
\bibliography{refs.bib}

\begin{aicauthors}
\begin{authorinfo}[antonio]
Ant\'onio Gir\~ao\\
Department of Mathematics\\
University College London\\
London WC1E 6BT, UK\\
a\imagedot{}girao\imageat{}ucl\imagedot{}ac\imagedot{}uk
\end{authorinfo}
\begin{authorinfo}[zach]
Zach Hunter\\
Department of Mathematics\\
ETH Z\"urich\\
8092 Z\"urich, Switzerland\\
zach\imagedot{}hunter\imageat{}math\imagedot{}ethz\imagedot{}ch
\end{authorinfo}
\begin{authorinfo}[yuval]
Yuval Wigderson\\
Institute for Theoretical Studies\\
ETH Z\"urich\\
8006 Z\"urich, Switzerland\\
yuval\imagedot{}wigderson\imageat{}eth-its\imagedot{}ethz\imagedot{}ch
\end{authorinfo}
\end{aicauthors}

\end{document}